\documentclass[11pt,a4paper]{article}
\usepackage{amsmath}
\usepackage{amsthm}
\usepackage{amscd}
\usepackage{amsfonts}
\usepackage{amssymb}
\usepackage{graphicx}
\usepackage{color}
\usepackage{verbatim}
\usepackage{color}
\usepackage{subfigure}
\usepackage{float}
\usepackage{authblk}
\usepackage[colorlinks]{hyperref}
\usepackage{textcomp}
\usepackage{eurosym}
\usepackage[none]{hyphenat} 

\usepackage{textcomp}
\usepackage{eurosym}
\usepackage{mathtools}
\usepackage{enumerate}
\usepackage{amscd} 
\usepackage{xy}
\xyoption{all}

\newtheorem{theorem}{Theorem}[section]
\newtheorem{lemma}[theorem]{Lemma}

\newtheorem{corollary}[theorem]{Corollary}
\theoremstyle{definition}

\begin{document}
\title{Logarithmically Complete Monotonicity of Certain Ratios Involving the $k$-Gamma Function}
\author[1]{Kwara Nantomah}
\author[2]{Li Yin}

\affil[1]{Department of Mathematics, University for Development Studies, Navrongo Campus, P. O. Box 24, Navrongo, UE/R, Ghana (E-mail: knantomah@uds.edu.gh)}
\affil[2]{Department of Mathematics, Binzhou University, Binzhou City, Shandong Province, P.O.Box 256603 China (E-mail: yinli\_79@163.com)}

\date{\small Submitted to RGN Publications on: 17-October-2018, Accepted on: 26-November-2018}

\maketitle
\begin{abstract} In this paper, we prove logarithmically complete monotonicity properties of certain ratios of the $k$-gamma function.  As a consequence, we deduce some inequalities involving the $k$-gamma and $k$-trigamma functions.

\vskip1em \noindent \textbf{2010 AMS Classification: 33B15, 26A48, 26A51}


\vskip1em \noindent \textbf{Article type:} (Research Article) \ \

\end{abstract}


\section{Introduction}

The $k$-gamma function (also known as the $k$-analogue or $k$-extension of the classical Gamma function) was defined by D\'{i}az and Pariguan \cite{Diaz-Pariguan-2007-DM}  for $k>0$ and $x\in \mathbb{C}\backslash k\mathbb{Z}$  as  
\begin{equation}\label{eqn:k-Gamma}
\Gamma_k(x)= \lim_{n\rightarrow \infty}\frac{n!k^{n}(nk)^{\frac{x}{k}-1}}{(x)_{n,k}} =  \int_0^\infty t^{x-1}e^{-\frac{t^k}{k}}\,dt,
\end{equation}
where \, $(x)_{n,k}=x(x+k)(x+2k)\dots (x+(n-1)k)$\, is the Pochhammer $k$-symbol. It satisfies the basic properties
\begin{align*}
\Gamma_{k}(x+k)&=x\Gamma_{k}(x),  \\
\Gamma_{k}(k)&=1 , \\
\Gamma_{k}(nk)&=k^{n-1}(n-1)!, \quad n\in \mathbb{N}. 
\end{align*}

The $k$-analogue of the Gauss multiplication formula is given as \cite{Rehman-et-al-2015-OM}
\begin{equation}\label{eqn:k-Gauss-Mult}
\Gamma_{k}(mx)=m^{\frac{mx}{k}-\frac{1}{2}} k^{\frac{m-1}{2}} (2\pi)^{\frac{1-m}{2}}\prod_{s=0}^{m-1}\Gamma_{k}\left(x+\frac{sk}{m} \right), \quad m\geq2 ,
\end{equation}
and by letting $x=\frac{k}{m}$, one obtains the identity
\begin{equation}\label{eqn:k-Euler-Type}
\prod_{s=1}^{m-1}\Gamma_{k}\left( \frac{sk}{m}\right) = \frac{k^{\frac{1-m}{2}} (2\pi)^{\frac{m-1}{2}}}{\sqrt{m}}, \quad m\geq2  .
\end{equation}

The logarithmic derivative of the $k$-gamma function, which is called the $k$-digamma function, is defined as  (see \cite{Kokologiannaki-Sourla-2013-JIASF}, \cite{Krasniqi-2010-SM}, \cite{Mubeen-Iqbal-2015-AIA}, \cite{Mubeen-Rehman-Shaheen-2014-BJ} )
\begin{align}
\psi_{k}(x)=\frac{d}{dx}\ln \Gamma_{k}(x)&=\frac{\ln k-\gamma}{k}-\frac{1}{x}+\sum_{n=1}^{\infty} \frac{x}{nk(nk+x)} \nonumber \\ 
&=\frac{\ln k-\gamma}{k}+\sum_{n=0}^{\infty} \left( \frac{1}{nk+k} - \frac{1}{nk+x} \right), \label{eqn:k-Digamma-Series-Rep-2}
\end{align}
and satifies the properties 
\begin{align*}
\psi_{k}(x+k)&= \frac{1}{x} + \psi_{k}(x) ,  \\
\psi_{k}(k)&= \frac{\ln k - \gamma}{k},
\end{align*}
where $\gamma$  is the Euler-Mascheroni's constant. The $k$-polygamma function of order $r$ is defined as \cite{Krasniqi-2010-SM}
\begin{equation}\label{eqn:k-Polygamma-Series}
\psi_{k}^{(r)}(x)=\frac{d^r}{dx^r}\psi_{k}(x)= (-1)^{r+1} \sum_{n=0}^{\infty}\frac{r!}{(nk+x)^{r+1}}, \quad r\in \mathbb{N}.
\end{equation}

Also, it is well known in the literature that the integral
\begin{equation}\label{eqn:integral-rep-vip}
\frac{r!}{x^{r+1}} = \int_{0}^{\infty}t^{r}e^{-xt}\,dt,
\end{equation}
holds for $x>0$ and $r\in \mathbb{N}_0=\mathbb{N}\cup \{0\}$. See for instance \cite[p. 255]{Abramowitz-Stegun-1972}.

We recall that a function $h$ is said to be completely monotonic on an interval $I$, if $h$ has derivatives of all order and satisfies
\begin{equation}\label{eqn:def-complete-monotonicity}
(-1)^{r}h^{(r)}(z)\geq0,
\end{equation}
for all $z\in I$ and $r\in \mathbb{N}_0$. If inequality \eqref{eqn:def-complete-monotonicity} is strict, then $f$ is said to be strictly completely monotonic on $I$. In particular, each completely monotonic function is positive, decreasing and convex.

A positive function $h$ is said to be logarithmically completely monotonic on an interval $I$, if $h$ satisfies \cite{Qi-Guo-2004-RGMIA}
\begin{equation}\label{eqn:def-log-complete-monotonicity}
(-1)^{r}[\ln h(z)]^{(r)} \geq0, 
\end{equation}
for all $z\in I$ and $r\in \mathbb{N}_0$. If inequality \eqref{eqn:def-log-complete-monotonicity} is strict, then $h$ is said to be strictly logarithmically completely monotonic on $I$. It has been established that, if a function is logarithmically completely
monotonic, then  it is completely monotonic \cite{Qi-Guo-2004-RGMIA}. However, the converse is not necessarily true.

Completely monotonic functions are frequently encounted in various aspects of mathematics. They are particularly useful in the theory of inequalities,  in probability theory and in potential theory. There exists an extensive literature on this subject.
See for instance \cite{Alzer-Berg-2002-AASF}, \cite{Chen-Choi-2017-MIA}, \cite{Merkle-2014-ANTATSF}, \cite{Miller-Samko-2001-ITSF}, \cite{Yin-Huang-Lin-Wang-2018-ADE}, \cite{Yin-Huang-Song-Duo-2018-JIA} and the related references therein.

In \cite{Merkle-1997-SM} it was shown that the functions
\begin{equation*}
F(x)=\frac{\Gamma(2x)}{x\Gamma^{2}(x)} \quad \text{and} \quad G(x)=\frac{\Gamma(2x)}{\Gamma^{2}(x)} ,
\end{equation*}
are strictly logarithmically convex and strictly logarithmically concave respectively on $(0,\infty)$. In \cite{Chen-2005-SM}, the author established a more deeper results by proving that $F(x)$ and $1/G(x)$ are strictly logarithmically completely monotonic on  $(0,\infty)$. Then in \cite{Li-Zhao-Chen-2006-SM}, the authors generalized the results of \cite{Chen-2005-SM} by proving the following results.

Let $F$ and $G$ be defined for an integer $m\geq2$ and $x\in(0,\infty)$ as
\begin{equation*}
F(x)=\frac{\Gamma(mx)}{x^{m-1}\Gamma^{m}(x)} \quad \text{and} \quad G(x)=\frac{\Gamma(mx)}{\Gamma^{m}(x)} .
\end{equation*}
Then $F(x)$ and $1/G(x)$ are strictly logarithmically completely monotonic on  $(0,\infty)$.

In this paper the main objective is to extend the results of \cite{Li-Zhao-Chen-2006-SM} to the $k$-gamma function. We begin with the following auxiliary results.

\section{Auxiliary Results}

\begin{lemma}\label{lem:Mult-Formula-k-Polygamma}
The $k$-polygamma function satisfies the identity
\begin{equation}\label{eqn:Mult-Formula-k-Polygamma}
\psi_{k}^{(r)}(mx)= \frac{1}{m^{r+1}}\sum_{s=0}^{m-1}\psi_{k}^{(r)} \left( x+\frac{sk}{m}\right),  \quad m\geq2,
\end{equation}
where $r\in \mathbb{N}$. This may be called the multiplication formula for the $k$-polygamma function.
\end{lemma}

\begin{proof}
This follows easily from \eqref{eqn:k-Gauss-Mult}.
\end{proof}

\begin{lemma}\label{lem:Integral-Reps-k-digamma-k-Polygamma}
The $k$-digamma and $k$-polygamma functions satisfy the following integral representations.
\begin{align}
\psi_{k}(x)&=\frac{\ln k - \gamma}{k} + \int_{0}^{\infty} \frac{e^{-kt}-e^{-xt}}{1-e^{-kt}}\,dt  \label{eqn:k-digamma-Int-Rep-1} \\
&=\frac{\ln k - \gamma}{k} + \int_{0}^{1} \frac{t^{k-1}-t^{x-1}}{1-t^{k}}\,dt, \label{eqn:k-digamma-Int-Rep-2} \\
\psi_{k}^{(r)}(x)&= (-1)^{r+1} \int_{0}^{\infty} \frac{t^{r}e^{-xt}}{1-e^{-kt}}\,dt  \label{eqn:k-polygamma-Int-Rep-1} \\
&= - \int_{0}^{1} \frac{(\ln t)^{r}t^{x-1}}{1-t^{k}}\,dt.  \label{eqn:k-polygamma-Int-Rep-2}
\end{align}
\end{lemma}

\begin{proof}
By using \eqref{eqn:k-Digamma-Series-Rep-2} in conjunction with \eqref{eqn:integral-rep-vip}, we obtain
\begin{align*}
\psi_{k}(x)&=\frac{\ln k - \gamma}{k} + \sum_{n=0}^{\infty}\int_{0}^{\infty}\left( e^{-kt} - e^{-xt}\right)e^{-nkt}\,dt \\
&=\frac{\ln k - \gamma}{k} + \int_{0}^{\infty}\left( e^{-kt} - e^{-xt}\right)\sum_{n=0}^{\infty}e^{-nkt}\,dt \\
&=\frac{\ln k - \gamma}{k} + \int_{0}^{\infty} \frac{e^{-kt}-e^{-xt}}{1-e^{-kt}}\,dt,
\end{align*}
which proves \eqref{eqn:k-digamma-Int-Rep-1}, and by change of variables, we obtain \eqref{eqn:k-digamma-Int-Rep-2}. Representations  \eqref{eqn:k-polygamma-Int-Rep-1} and  \eqref{eqn:k-polygamma-Int-Rep-2} respectively follow directly from \eqref{eqn:k-digamma-Int-Rep-1} and \eqref{eqn:k-digamma-Int-Rep-2}.
\end{proof}

\begin{lemma}\label{lem:Sum-Exp-vrs-Exp}
For $t>0$ and $n\in \mathbb{N}$, we have
\begin{equation}\label{eqn:Sum-Exp-vrs-Exp}
\sum_{s=1}^{n}e^{-\frac{st}{n+1}} - ne^{-t}>0 .
\end{equation}
\end{lemma}

\begin{proof}
Note that $e^{-ut}>e^{-t}$ for all $0<u<1$ and $t>0$. Then we have
\begin{align*}
\sum_{s=1}^{n}e^{-\frac{st}{n+1}} &= e^{-\frac{1}{n+1}t} + e^{-\frac{2}{n+1}t} + e^{-\frac{3}{n+1}t} + \dots + e^{-\frac{n}{n+1}t} \\
&> e^{-t} + e^{-t} + e^{-t} + \dots + e^{-t} \\
&=ne^{-t},
\end{align*}
which completes the proof.
\end{proof}

\section{Main Results}

We are now in position to prove the main results of this paper.

\begin{theorem}\label{thm:Generalized-LCM-k-Gamma}
Let $F$ and $G$ be defined for an integer $m\geq2$, $k>0$ and $x\in(0,\infty)$ as
\begin{equation*}
F(x)=\frac{\Gamma_{k}(mx)}{x^{m-1}\Gamma_{k}^{m}(x)} \quad \text{and} \quad G(x)=\frac{\Gamma_{k}(mx)}{\Gamma_{k}^{m}(x)} .
\end{equation*}
Then
\begin{enumerate}[(a)]
\item $F(x)$ is strictly logarithmically completely monotonic on  $(0,\infty)$,
\item $1/G(x)$ is strictly logarithmically completely monotonic on  $(0,\infty)$.
\end{enumerate}
\end{theorem}

\begin{proof}
By repeated differentiations and applying \eqref{eqn:Mult-Formula-k-Polygamma}, we obtain
\begin{align*}
(\ln F(x))^{(r)}&=m^{r}\psi_{k}^{(r-1)}(mx) - m\psi_{k}^{(r-1)}(x) + (-1)^{r} \frac{(m-1)(r-1)!}{x^{r}} \\
& =\sum_{s=0}^{m-1}\psi_{k}^{(r-1)} \left( x+\frac{sk}{m}\right) - m\psi_{k}^{(r-1)}(x) + (-1)^{r} \frac{(m-1)(r-1)!}{x^{r}}, 
\end{align*}
for $r\in \mathbb{N}$. Then by applying \eqref{eqn:integral-rep-vip} and \eqref{eqn:k-polygamma-Int-Rep-1}, we obtain  
\begin{align*}
(-1)^{r}(\ln F(x))^{(r)} 
&=\int_{0}^{\infty}\left[ \sum_{s=0}^{m-1}e^{-\frac{sk}{m}t} - m + (m-1)(1-e^{-kt})  \right]\frac{t^{r-1}e^{-xt}}{1-e^{-kt}}\,dt \\
&=\int_{0}^{\infty}\left[ \sum_{s=1}^{m-1}e^{-\frac{sk}{m}t} - (m-1)e^{-kt}  \right]\frac{t^{r-1}e^{-xt}}{1-e^{-kt}}\,dt \\
&>0,
\end{align*}
which follows from Lemma \ref{lem:Sum-Exp-vrs-Exp}. This completes the proof of (a). Similarly, we obtain
\begin{align*}
\left( \ln \frac{1}{G(x)}\right)^{(r)}&=m\psi_{k}^{(r-1)}(x) - m^{r}\psi_{k}^{(r-1)}(mx) \\
&=m\psi_{k}^{(r-1)}(x) - \sum_{s=0}^{m-1}\psi_{k}^{(r-1)} \left( x+\frac{sk}{m}\right),  
\end{align*}
for $r\in \mathbb{N}$. This implies that
\begin{align*}
(-1)^{r} \left( \ln \frac{1}{G(x)}\right)^{(r)}
&=\int_{0}^{\infty}\left[ m - \sum_{s=0}^{m-1}e^{-\frac{sk}{m}t}\right]\frac{t^{r-1}e^{-xt}}{1-e^{-kt}}\,dt \\
&=\int_{0}^{\infty}\left[ 1 + \sum_{s=1}^{m-1}1 - \sum_{s=0}^{m-1}e^{-\frac{sk}{m}t}\right]\frac{t^{r-1}e^{-xt}}{1-e^{-kt}}\,dt \\
&=\int_{0}^{\infty}\sum_{s=1}^{m-1}\left(1 - e^{-\frac{sk}{m}t}\right)\frac{t^{r-1}e^{-xt}}{1-e^{-kt}}\,dt \\
&>0, 
\end{align*}
which completes the proof of (b). 
\end{proof}

\begin{corollary}\label{cor:Ineq-Ratio-k-Gamma}
Let $m\geq2$ be an integer and $k>0$. Then the inequality
\begin{equation}\label{eqn:Ineq-Ratio-k-Gamma}
k^{m-1}(m-1)! < \frac{\Gamma_{k}(mx)}{\Gamma_{k}^{m}(x)} < x^{m-1}(m-1)! , 
\end{equation}
holds if $x\in(k,\infty)$ and reverses if $x\in(0,k)$.
\end{corollary}

\begin{proof}
The conclusions of Theorem \ref{thm:Generalized-LCM-k-Gamma} imply that $F(x)$ is strictly decreasing while $G(x)$ is strictly increasing. Then for $x\in(k,\infty)$, we have $F(x)<F(k)$ which gives the right hand side of \eqref{eqn:Ineq-Ratio-k-Gamma}. Also for  $x\in(k,\infty)$ we have $G(x)>G(k)$  yielding the left hand side of \eqref{eqn:Ineq-Ratio-k-Gamma}.
The proof for the case where $x\in(0,k)$ follows the same procedure and so we omit the details.
\end{proof}

\begin{corollary}\label{cor:Ineq-Ratio-k-Gamma-2}
Let $m\geq2$ be an integer and $k>0$. Then the inequality
\begin{equation}\label{eqn:Ineq-Ratio-k-Gamma-2}
\frac{\Gamma_{k}(mx)}{\Gamma_{k}^{m}(x)} < \frac{x^{m-1}}{m},
\end{equation}
holds for $x\in(0,\infty)$ .
\end{corollary}

\begin{proof}
Since $F(x)$ is strictly decreasing on $(0, \infty)$, it follows easily that
\begin{equation*}
F(x)<F(0)=\lim_{x\rightarrow 0}F(x)=\frac{1}{m},
\end{equation*}
which gives \eqref{eqn:Ineq-Ratio-k-Gamma-2}. 
\end{proof}

\begin{corollary}\label{cor:Ineq-Diff-k-Trigamma}
Let $m\geq2$ be an integer and $k>0$. Then the inequality
\begin{equation}\label{eqn:Ineq-Diff-k-Trigamma}
 \frac{1}{m}\sum_{s=0}^{m-1}\psi_{k}' \left( x+\frac{sk}{m}\right) <
\psi_{k}'(x) < 
 \frac{1}{m}\sum_{s=0}^{m-1}\psi_{k}' \left( x+\frac{sk}{m}\right) + \frac{m-1}{mx^2},
\end{equation}
holds for $x\in(0,\infty)$.
\end{corollary}

\begin{proof}
Infering from Theorem \ref{thm:Generalized-LCM-k-Gamma}, it follows that $F(x)$ is strictly logarithmically convex while $G(x)$ is strictly logarithmically concave. In this way,
\begin{equation*}
\left( \ln F(x)\right)''= \sum_{s=0}^{m-1}\psi_{k}' \left( x+\frac{sk}{m}\right) - m\psi_{k}'(x) +  \frac{m-1}{x^2} >0,
\end{equation*}
which yields the right hand side of \eqref{eqn:Ineq-Diff-k-Trigamma}. Also,
\begin{equation*}
\left( \ln G(x)\right)''=  \sum_{s=0}^{m-1}\psi_{k}' \left( x+\frac{sk}{m}\right) - m\psi_{k}'(x)  <0,
\end{equation*}
which gives the left hand side of \eqref{eqn:Ineq-Diff-k-Trigamma}. This completes the proof.
\end{proof}

\section{Conclusion}

We have established logarithmically complete monotonicity properties of certain ratios of the $k$-gamma function.  As a consequence, we derived some inequalities involving the $k$-gamma and the $k$-trigamma functions. The established results could trigger a new research direction in the theory of inequalities and special functions.

\subsection*{Acknowledgement}
The authors are grateful to the anonymous referees for thorough reading of the manuscript.

\subsection*{Competing Interests}
The authors declare that they have no competing interests.

\subsection*{Authors' Contributions}
All the authors contributed significantly in writing this article. The authors read and approved
the final manuscript.



\begin{thebibliography}{99}

\bibitem{Abramowitz-Stegun-1972} M. Abramowitz and I. A. Stegun (Eds), \textit{Handbook of Mathematical Functions with Formulas, Graphs, and Mathematical Tables}, National Bureau of Standards, Applied Mathematics Series 55, 10th Printing, Washington, 1972. \url{http://people.math.sfu.ca/\~cbm/aands/abramowitz\_and\_stegun.pdf}

\bibitem{Alzer-Berg-2002-AASF} H. Alzer and C. Berg, Some classes of completely monotonic functions, \textit{Ann. Acad. Scient. Fennicae}, \textbf{27}(2002), 445-460. \url{https://www.acadsci.fi/mathematica/Vol27/alzer.pdf}

\bibitem{Chen-2005-SM} C-P. Chen, Complete monotonicity properties for a ratio of gamma functions, \textit{Ser. Mat.}, \textbf{16}(2005), 26-28. \url{https://www.jstor.org/stable/pdf/43666610.pdf}

\bibitem{Chen-Choi-2017-MIA} C-P. Chen and J. Choi, Completely monotonic functions related to Gurland's ratio for the gamma function,  \textit{Math. Inequal. Appl.}, \textbf{20}(3)(2017), 651-659. DOI: 10.7153/mia-20-43

\bibitem{Diaz-Pariguan-2007-DM} R. D\'{i}az and E. Pariguan, On hypergeometric functions and Pochhammer $k$-symbol, \textit{Divulg. Mat.}, \textbf{15}(2007) 179-192. \url{https://www.emis.de/journals/DM/v15-2/art8.pdf}

\bibitem{Kokologiannaki-Sourla-2013-JIASF} C. G. Kokologiannaki and V. D. Sourla, Bounds for $k$-gamma and $k$-beta functions, \textit{J. Inequal. Spec. Funct.}, \textbf{4}(3)(2013), 1-5. \url{http://46.99.162.253:88/jiasf/repository/docs/JIASF4-3-1.pdf}

\bibitem{Krasniqi-2010-SM} V. Krasniqi, Inequalities and monotonicity for the ration of $k$-gamma functions, \textit{Scientia Magna}, \textbf{6}(1)(2010), 40-45. \url{http://fs.unm.edu/ScientiaMagna6no1.pdf#page=46}

\bibitem{Li-Zhao-Chen-2006-SM} A-J. Li, W-Z. Zhao and C-P. Chen, Logarithmically complete monotonicity and Shur- convexity for some ratios of gamma functions, \textit{Ser. Mat.}, \textbf{17}(2006), 88-92. \url{https://www.jstor.org/stable/43660752}

\bibitem{Merkle-1997-SM} M. Merkle, On log-convexity of a ratio of gamma functions, \textit{Ser. Mat.}, \textbf{8}(1997), 114-119. \url{https://www.jstor.org/stable/43666390}

\bibitem{Merkle-2014-ANTATSF} M. Merkle, \textit{Completely monotone functions: A digest}, In: Milovanović G., Rassias M. (eds) Analytic Number Theory, Approximation Theory, and Special Functions. Springer, New York, NY., (2014), 347-364. DOI: 10.1007/978-1-4939-0258-3\_12.

\bibitem{Miller-Samko-2001-ITSF} K. S. Miller and S. G. Samko, Completely monotonic functions, \textit{Integr. Transf. and Spec. Funct.}, \textbf{12}(4)(2001), 389-402. DOI: 10.1080/10652460108819360

\bibitem{Mubeen-Iqbal-2015-AIA} S. Mubeen and S. Iqbal, Some inequalities for the gamma $k$-function, \textit{Adv. Inequal. Appl.}, \textbf{2015}(2015), At ID: 10, 1-9. \url{http://www.scik.org/index.php/aia/article/view/2252}

\bibitem{Mubeen-Rehman-Shaheen-2014-BJ} S.  Mubeen,  A.  Rehman and F. Shaheen, Properties of $k$-gamma, $k$-beta and $k$-psi functions, \textit{Bothalia Journal}, \textbf{44}(2014) 372-380. 

\bibitem{Qi-Guo-2004-RGMIA} F. Qi and B-N. Guo, Complete monotonicities of functions involving the gamma and digamma functions, \textit{RGMIA Res. Rep. Coll.}, \textbf{7}(1)(2004), Art. 6. \url{http://www.rgmia.org/papers/v7n1/minus-one.pdf}

\bibitem{Rehman-et-al-2015-OM} A. Rehman, S. Mubeen, R. Safdar and N. Sadiq, Properties of $k$-beta function with several variables, \textit{Open Math.}, \textbf{13}(1)(2015), 308-320. DOI: 10.1515/math-2015-0030

\bibitem{Yin-Huang-Song-Duo-2018-JIA} L. Yin, L-G. Huang, Z-M. Song and X. K. Dou, Some monotonicity properties and inequalities for the generalized digamma and polygamma functions, \textit{J. Inequal. Appl.}, \textbf{2018}(2018), Art ID:249, 13 pages. DOI: 10.1186/s13660-018-1844-2

\bibitem{Yin-Huang-Lin-Wang-2018-ADE} L. Yin, L-G. Huang, X-L. Lin and Y-L. Wang, Monotonicity, concavity, and inequalities related to the generalized digamma function, \textit{Adv. Difference Equ.}, \textbf{2018}(2018), Art ID:246, 9 pages. DOI: 10.1186/s13662-018-1695-7

\end{thebibliography}
\end{document}